\documentclass{amsart}
\usepackage[T1]{fontenc}
\usepackage{amsmath,amssymb,amsthm}
\usepackage[margin = 1 in]{geometry}
\usepackage{array}
\usepackage{tikz}
\usepackage{tikz-cd}
\usepackage{mathrsfs}
\usepackage{enumitem}
\usepackage{mathtools}
\usepackage{xspace}
\usepackage{graphicx}
\usepackage{chngcntr}
\usepackage{colonequals}
\usepackage{parskip}

\newtheorem{theorem}[subsection]{Theorem}
\newtheorem{lemma}[subsection]{Lemma}
\newtheorem{cor}[subsection]{Corollary}

\theoremstyle{definition}
\newtheorem{defn}[subsection]{Definition}

\theoremstyle{remark}
\newtheorem{remark}[subsection]{Remark}

\newcommand{\Z}{\mathbb{Z}}

\newcommand{\Q}{\mathbb{Q}}
\newcommand{\Fp}[1][p]{\mathbb{F}_{#1}}
\newcommand{\N}{\mathbb{N}}

\newcommand{\<}{\left\langle}
\renewcommand{\>}{\right\rangle}

\newcommand{\dott}[2]{#1 _1, \dots , #1 _{#2}}

\DeclareMathOperator{\fpt}{fpt}
\DeclareMathOperator{\lcm}{lcm}
\DeclareMathOperator{\lct}{lct}

\title{On Stability and Denominators of $F$-pure Thresholds in Families of Diagonal Hypersurfaces}
\author{Gari Lincoln Chua}
\date{}
\address{Department of Mathematics and Statistics, University of New Mexico, 1 University of New Mexico, MSC01 1115\\
	Albuquerque, NM 87131, United States of America}
\email{garichua@unm.edu}

\begin{document}
	
	\begin{abstract}
		Given a prime number $p$ and a positive integer $m$, we provide a family of diagonal hypersurfaces $\{ f_n \}_{n = 1}^{\infty}$ in $m$ variables, for which the denominator of $\fpt (f_{n})$ (in lowest terms) is always $p$ and whose $F$-pure thresholds stabilize after a certain $n$. We also provide another family of diagonal hypersurfaces $\{ g_n \}_{n = 1}^{\infty}$ in $m$ variables, for which the power of $p$ in the denominator of $\fpt (g_{n})$ (in lowest terms) diverges to $\infty$ as $n \to \infty$. This behavior of the denominator of the $F$-pure thresholds is dependent on the congruence class of $p$ modulo the smallest two exponents of $\{ f_n \}$ and $\{ g_n \}$. 
	\end{abstract}
	
	\maketitle
	
	\section{Introduction}
	
	Let $R = k[x_1, \dots , x_n]$ where $k$ is a perfect field of characteristic $p > 0$ and denote the maximal ideal $\mathfrak{m} = (x_1, \dots , x_n)$. The $F$-pure threshold of a polynomial $f \in R$, introduced in \cite{TW04}, is defined to be $\fpt (f) = \sup_{e} \left\{ \frac{N}{p^e} : f^N \not \in \mathfrak{m}^{[p^e]} \right\}$, where $\mathfrak{m}^{[p^e]} = (x_1^{p^e}, \dots , x_n^{p^e})$ is the $e$th Frobenius power of $\mathfrak{m}$. The $F$-pure threshold $\fpt (f)$ is a rational number in $(0, 1]$ that is an invariant of positive characteristic singularities, with smaller values meaning that the hypersurface is ``more singular."
	
	\bigskip 
	
	The $F$-pure threshold (fpt) is closely related to the log canonical threshold (lct), an invariant of characteristic $0$ singularities. If $f \in \mathbb{Q} [x_1, \dots , x_n]$ and $\overline{f}$ is the reduction of $f$ modulo $p$, both $\lct (f)$ and $\fpt (\overline{f})$ are rational numbers between $0$ and $1$. Moreover,
	\[ \fpt (\overline{f}) \le \lct (f) \hspace{10 pt} \text{ and } \hspace{10 pt} \lim_{p \to \infty} \fpt (\overline{f}) = \lct (f). \] 
	This motivates the following question (see for example, \cite{Sch08}): When is $\fpt (\overline{f}) = \lct (f)$, and if $\fpt (\overline{f}) \neq \lct (f)$, does $p$ always divide the denominator of $\fpt (\overline{f})$?
	
	\bigskip
	
	The answer to the previous question is complicated. Canton, Hernandez, Schwede, and Witt in \cite{CHSW16} have shown a family of polynomials $f$ for which $\fpt (\overline{f}) \neq \lct (f)$ but the denominator of $\fpt (\overline{f})$ is not a multiple of $p$. On the other hand, \cite{Her14-1}, \cite{HNWZ16}, \cite{Mul17} have shown that in many classes of polynomials, if $\fpt (\overline{f}) \neq \lct (f)$, then the denominator of $\fpt (\overline{f})$ must be a multiple of $p$. In particular, in \cite{HNWZ16}, the authors have shown that if $f$ is a quasi-homogeneous polynomial with an isolated singularity at the origin, then $\fpt (\overline{f})$ is either $\lct (f)$ or its denominator must be a power of $p$. The same paper also gives a bound on the power of $p$ (depending on $\overline{f}$) appearing in the denominator of $\fpt (\overline{f})$ depending on the choice of grading on $k [x_1, \dots , x_n]$ that turns $\overline{f}$ homogeneous.
	
	\bigskip
	The aformentioned bound in the power of $p$ in the $F$-pure threshold motivates the following question: Given a family of quasi-homogeneous polynomials $\{ f_n \}$ in characteristic $p$ for which $\fpt (f_n) \neq \lct (f_n)$ for all $n$, is there a positive integer $M$ such that the power of $p$ in the denominator of $\fpt (f_n)$ is less than or equal to $M$ for all $n$?
	
	\bigskip
	In this paper, we consider families of diagonal hypersurfaces; i.e., hypersurfaces cut out by polynomials of the form $x_1^{a_1} + \dots + x_m ^{a_m}$, where $m > 2$, the exponents $a_1$ and $a_2$ are always fixed, while the other exponents may vary across the natural numbers. We will see that the power of $p$ in the $F$-pure thresholds of these families is either bounded by 1, or diverges to $\infty$ depending on whether $p \equiv -1$ or $1 \mod \lcm (a_1, a_2)$. This provides some interesting behavior of $F$-pure thresholds. Our main results are as follows:
	
	\begin{theorem}
		Let $f_n (\dott{x}{m}) = x_1^{a_1} + x_2^{a_2} + \dots + x_{m - 1}^{a_{m - 1}} + x_m^{n} \in k [ \dott{x}{m}]$, where $1 < a_1 \le a_2 \le \dots \le a_{m - 1}$ are fixed integers, $n \ge a_{m - 1}$, and $(a_1, a_2) \neq (2, 2)$. If $p \geq  \frac{ a_1a_2\left( a_1 + a_2 + \frac{1}{a_1} + \frac{1}{a_2} - 4 \right)}{a_1a_2 - a_1 - a_2}$ and $p \equiv -1 \mod \lcm (a_1, a_2)$, then either $\fpt (f_n) = 1$ or the denominator of $\fpt (f_n)$ (in lowest terms) is $p$. Moreover, $\fpt (f_n) = \fpt (x_1^{a_1} + x_2^{a_2} + x_3^{a_3} + \dots + x_{m - 1}^{a_{m - 1}} + x_m^{p} )$ for all $n \geq p$.
	\end{theorem}
	
	\begin{theorem}
		Let $f_n( \dott{x}{m}) = x_1^{a} + x_2^b + x_3^{n} + x_4^{n} + \dots + x_m ^n$, where $1 < a \leq b$ are fixed integers, integers, $n \ge b$, and $(a, b) \neq (2, 2)$. If $p \equiv 1 \mod \lcm (a, b)$ and $p^d \leq n < p^{d + 1}$ for some positive integer $d$ and $\fpt (f_n) \neq \frac{1}{a} + \frac{1}{b} + \frac{m - 2}{n}$, then the power of $p$ in the denominator of $\fpt (f_n)$ (in lowest terms) is at least $d$.
	\end{theorem}
	
	\begin{theorem}
		Let $s$ be a positive integer between 2 and $p - 1$ inclusive. For $n \ge 2$, let $f_n( \dott{x}{m}) = x_1^{a} + x_2^b + x_3^{p^n - s} + x_4^{p^n - s} + \dots + x_m ^{p^n - s}$, where $1 < a \leq b$ are integers and $(a, b) \neq (2, 2)$. Let $r = \left\lceil \log _{s} \left(  \left( 1 - \frac{1}{a} - \frac{1}{b} \right) (p - 1)  \right) \right\rceil$. If $p \equiv 1 \mod \lcm (a, b)$ and $s^r < p - 1$, then the power of $p$ in the denominator of $\fpt (f_n)$ is exactly $d \left( r + 1 \right) - 1$.
	\end{theorem}
	
	\begin{cor}
		Given a prime $p > 5$ and positive integers $\ell, m$ with $m > 2$, there exists a family of diagonal hypersurfaces $\{ f_n \}_{n \in \N}$ in $m$ variables such that the power of $p$ in the denominator of $\fpt (f_n)$ is less than or equal to $\ell$ for all $n \in \N$, and $\fpt (f_n) = \fpt (f_{p^{\ell}})$ for all $n \ge p^{\ell}$. There also exists a different family of diagonal hypersurfaces $\{ g_n \}_{n \in \N}$ such that the power of $p$ in the denominator of $\fpt (g_n)$ diverges to $\infty$ as $n \to \infty$.
	\end{cor}
	
	\textit{Acknowledgements.} The author thanks Dr. Janet Vassilev for the many insightful discussions.
	
	\section{Nonterminating Base $p$ Expansions}

	In this section, we review relevant results regarding nonterminating base $p$ expansions. The content of this section is lifted from \cite{Her14-1} and \cite{HNWZ16}.
	
	\begin{defn}
		Let $\alpha \in \Q \cap [0, 1]$ and $p$ be a prime number. The nonterminating base $p$ expansion of $\alpha$ is the expression $\alpha = \sum _{e = 1}^{\infty} \alpha ^{(e)} p^{-e}$, where for each $e$, $0 \le \alpha^{(e)} \le p - 1$, and the sequence $\{ \alpha^{(e)} \}_{e \in \N}$ is not eventually zero. The nonterminating base $p$ expansion of $\alpha$ is unique, and for each positive integer $e$, we call $\alpha^{(e)}$ the $e$-th digit of $\alpha$ (base $p$).
	\end{defn}
	
	\begin{defn}
		Let $\alpha \in \Q \cap [0, 1]$ and $p$ be a prime number. For $e \in \N$, define the $e$-th truncation of $\alpha$ (base $p$) as $\< \alpha \> _e = \alpha^{(1)} p^{-1} + \dots + \alpha ^{(e)} p^{-e}$, and for convenience, we may also write 
		\[ \< \alpha \>_e = . \alpha^{(1)} : \alpha^{(2)} : \dots : \alpha^{(e)} (\text{base } p). \]
		By convention, $\< \alpha \> _0 = 0$ and $\< \alpha \> _{\infty} = \alpha$.
	\end{defn}
	
	Some relevant properties of nonterminating base $p$ expansions are encoded in the following lemma:
	\begin{lemma}
		
		Fix $\alpha, \beta \in \Q \cap [0, 1]$.
		\begin{enumerate}[leftmargin = *, label = (\arabic *)]
			\item If $\< \alpha \> _e \neq 0$, then the denominator of $\< \alpha \> _e$ is less than or equal to $p^d$ for some $d \le e$.
			
			\item $\alpha \le \beta$ if and only if $\< \alpha \> _{e} \le \< \beta \> _e$ for all $e \ge 1$; if $\alpha < \beta$, these inequalities are strict for $e \gg 0$.
			
			\item If $(p^d - 1) \alpha \in \mathbb{Z}$ for some positive integer $d$, then the base $p$ expansion of $\alpha$ is periodic with period dividing $d$.
		\end{enumerate}
	\end{lemma}
	
	\begin{proof}
		(1) and (2) follow from the definition; see \cite[Lemmas 2.5 and 2.6]{HNWZ16}  for (3).
	\end{proof}
	
	\section{$F$-pure Thresholds of Diagonal and Quasi-homogeneous Polynomials}
	
	From this point onward, we set $k$ to be a perfect field of characteristic $p > 0$ (typically, $k = \Fp$), $R = k [\dott{x}{n}]$ the polynomial ring in $n$ variables over $k$, and $\mathfrak{m} = (\dott{x}{n}) \subseteq R$ be the maximal ideal of $R$ generated by the variables.
	
	\begin{defn}
		Let $f \in \mathfrak{m}$. For each positive integer $e$, define
		\[ \nu _f (p^e) = \sup \left\{ N \in \Z : f^{N} \notin \mathfrak{m}^{[p^e]} \right\} \]
		where $\mathfrak{m}^{[p^e]} = (\dott{x^{p^e}}{n})$ denotes the $e$-th Frobenius power of $\mathfrak{m}$. Define the $F$-pure threshold of $f$ as 
		\[ \fpt (f) = \lim _{e \to \infty} \frac{\nu_f (p^e)}{p^e}. \]
	\end{defn}
	
	\begin{remark}
		For any $f \in \mathfrak{m}$, the sequence $\{ p^{-e} \nu _f (p^e) \} _{e = 1}^{\infty}$ is a nondecreasing sequence of rational numbers contained in $(0, 1]$ (see \cite{MTW05}, Lemma 1.1 and Remark 1.2 for the proof). Therefore, the limit defining the $F$-pure threshold of $f$ exists, and the $F$-pure threshold is in $(0, 1]$.
	\end{remark}
	
	Developing formulas and algorithms for $F$-pure threshold for certain classes of polynomials has been an active area of research. For example, see \cite{Her14-1} for $F$-pure thresholds of diagonal hypersurfaces, \cite{Her14-2} for $F$-pure thresholds of binomial hypersurfaces, \cite{BS15} and \cite{Mul17} for $F$-pure thresholds of Calabi-Yau hypersurfaces, and \cite{SV23} for $F$-pure thresholds of generic homogeneous polynomials. We record some relevant results.
	
	\begin{theorem}[\cite{Her14-1} Theorem 3.4, Remark 3.5]
		\label{diag}
		Let $f \in k [ x_1, \dots, x_n]$ be a $k^*$-linear combination of the monomials $x_1^{d_1}, \dots , x_n^{d_n}$. Let $L = \inf \{ e \ge 0: (d_1^{-1})^{(e + 1)} + \dots + (d_n^{-1})^{(e + 1)} \ge p \}$.
		\begin{itemize}
			\item If $L < \infty$, then $\fpt (f) = \< d_1^{-1} \> _L + \dots + \< d_n ^{-1} \> _L + p^{-L}$.
			
			\item If $L = \infty$, then $\fpt (f) = d_1^{-1} + \dots + d_n^{-1}$.
		\end{itemize}
	\end{theorem}
	
	\begin{theorem}[\cite{HNWZ16}, Theorem 3.5, Corollary 3.9]
		\label{bound}
		Fix an $\mathbb{N}$-grading on the polynomial ring $k[x_1, \dots , x_n]$, and let $f$ be a homogeneous polynomial under this grading. Write $\lambda = \frac{\deg x_1 + \dots + \deg x_n}{\deg f} = \frac{s}{t}$ in lowest terms.
		\begin{itemize}
			\item If $\fpt (f) \neq \lambda$, then $\fpt (f) = \< \lambda \> _L - \frac{E}{p^L} $ for some pair $(L, E) \in \mathbb{N}^2$ satisfying $L \ge 1$ and $1 \le E \le n - 1 - \frac{\left \lceil [s p^L \% t] + s \right \rceil}{t}$, where $[m \% t]$ denotes the smallest positive residue of $m$ modulo $t$.
			
			\item If $\fpt (f) \neq \lambda$ and $p \not | t$, then the power of $p$ in the denominator of $\fpt (f)$ (in lowest terms) is less than or equal to $M = 2 \phi (t) + \left \lceil \log _2 (n - 1) \right \rceil$, where $\phi$ denotes Euler's totient function.
		\end{itemize}
	\end{theorem}
	
	These two important theorems tell us that the $F$-pure thresholds of certain classes of polynomials can be computed off nonterminating base $p$ expansions. In particular, the proofs of the main results will be based on finding $L$ and then using these two theorems.
	
	The second part of Theorem \ref{bound} says that the power of $p$ appearing in the denominator of the $F$-pure threshold of a quasi-homogeneous polynomial $f$ is bounded by a number depending on the $\N$-grading that turns $f$ homogeneous. This along with \cite{Mul17}, Theorem 4.3, motivates the premise of the paper: if we consider a family of quasi-homogeneous polynomials where one or more exponents vary through the family, what is the growth of the power of $p$ in the denominator of the $F$-pure thresholds?
	
	\section{Main Results}
	
	We now prove the main results.
	
	\begin{theorem}
		\label{stable}
		Let $f_n (\dott{x}{m}) = x_1^{a_1} + x_2^{a_2} + \dots + x_{m - 1}^{a_{m - 1}} + x_m^{n} \in k [ \dott{x}{m}]$, where $1 < a_1 \le a_2 \le \dots \le a_{m - 1}$ are fixed integers, $n \ge a_{m - 1}$, and $(a_1, a_2) \neq (2, 2)$. If $p \geq  \frac{ a_1a_2\left( a_1 + a_2 + \frac{1}{a_1} + \frac{1}{a_2} - 4 \right)}{a_1a_2 - a_1 - a_2}$ and $p \equiv -1 \mod \lcm (a_1, a_2)$, then either $\fpt (f_n) = 1$ or the denominator of $\fpt (f_n)$ (in lowest terms) is $p$. Moreover, $\fpt (f_n) = \fpt (x_1^{a_1} + x_2^{a_2} + x_3^{a_3} + \dots + x_{m - 1}^{a_{m - 1}} + x_m^{p} )$ for all $n \geq p$.
	\end{theorem}
	
	\begin{proof}
		Since $p \equiv -1 \mod \lcm (a_1, a_2)$, we must have $p \equiv -1 \mod a_1$ and $p \equiv - 1 \mod a_2$. Therefore, the first $2$ digits of the nonterminating base $p$ expansion of $\frac{1}{a_1}$ satisfies 
		\[ \left( \frac{1}{a_1} \right)^{(1)} = \frac{p - (a_1 - 1)}{a_1}, \hspace{10 pt} \left(  \frac{1}{a_1} \right)^{(2)} = \left\lfloor p \left( \frac{a_1 - 1}{a_1} \right) \right\rfloor \geq (a_1 - 1)\left( \frac{p - (a_1 - 1)}{a_1} \right), \] 
		and similarly for $\frac{1}{a_2}$. Now, $$(a_1 - 1)\left( \frac{p - (a_1 - 1)}{a_1} \right) + (a_2 - 1) \left( \frac{p - (a_2 - 1)}{a_2} \right) \geq p \iff  p \geq  \frac{ a_1 a_2\left( a_1 + a_2 + \frac{1}{a_1} + \frac{1}{a_2} - 4 \right)}{a_1a_2 - a_1 - a_2}.$$ Therefore, for these $p$, we have $L < 2$.
		On the other hand, 
		$$\left( \frac{1}{a_1} \right)^{(1)} + \left( \frac{1}{a_2} \right)^{(1)} + \left( \frac{1}{n} \right)^{(1)} \geq p \iff n \leq \frac{p}{(p + 2)\left( 1 - \frac{1}{a_1} - \frac{1}{a_2} \right)} < 1 - \frac{a_1 + a_2}{a_1a_2 - a_1 - a_2}.$$ 
		Since $(a_1, a_2) \neq (2, 2)$, $a_1a_2 > a_1 + a_2$, and so the last inequality implies that to ``carry at the first place," we must have $n < 1$. This is impossible. Therefore, $L \geq 1$.
		This then proves that $L = 1$, and so 
		\[  \fpt (f) = \< \frac{1}{a_1} \> _{1} + \< \frac{1}{a_2} \> _{1} + \dots + \< \frac{1}{a_{m - 1}} \> _{1} + \< \frac{1}{n} \> _{1} + \frac{1}{p} \]
		
		by Theorem \ref{diag}. Since $\left( \frac{1}{a_1} \right)^{(1)} + \left( \frac{1}{a_2} \right)^{(1)} = (p + 1) \left( \frac{1}{a_1} + \frac{1}{a_2} \right) - 2 < p$ for all $(a_1, a_2) \neq (2, 2)$, this proves the result.
		
		Finally, if $n \geq p$, then $\left( \frac{1}{n} \right)^{(1)} = 0$, and
		
		\[ \fpt (f) = \< \frac{1}{a_1} \> _{1} + \< \frac{1}{a_2} \> _{1} + \dots + \< \frac{1}{a_{m - 1}} \> _{1} + \frac{1}{p} = \fpt (x_1^{a_1} + x_2^{a_2} + x_3^{a_3} + \dots + x_{m - 1}^{a_{m - 1}} + x_m^{p} ). \]
	\end{proof}
	
	Theorem \ref{stable} shows that for certain families of hypersurfaces $\{f_n \}_{n \in \N}$, the $F$-pure thresholds stabilize as $n \to \infty$ and their denominators are always either 1 or $p$. On the other hand, by considering different primes and indexing, we can make the power of $p$ in the $F$-pure thresholds of a family of diagonal hypersurfaces become unbounded as $n \to \infty$. One such method to do this is in the next result.
	
	\begin{theorem}
		\label{slow}
		Let $f_n( \dott{x}{m}) = x_1^{a} + x_2^b + x_3^{n} + x_4^{n} + \dots + x_m ^n$, where $1 < a \leq b$ are fixed integers, integers, $n \ge b$, and $(a, b) \neq (2, 2)$. If $p \equiv 1 \mod \lcm (a, b)$ and $p^d \leq n < p^{d + 1}$ for some positive integer $d$ and $\fpt (f_n) \neq \frac{1}{a} + \frac{1}{b} + \frac{m - 2}{n}$, then the power of $p$ in the denominator of $\fpt (f_n)$ (in lowest terms) is at least $d$.
	\end{theorem}
	
	\begin{proof}
		For all $e \geq 1$, $$\left( \frac{1}{a} \right)^{(e)} + \left( \frac{1}{b} \right)^{(e)} = (p - 1) \left( \frac{1}{a} + \frac{1}{b} \right) < p$$ for all $(a, b)$. Moreover, $p^d \leq n < p^{d + 1}$ implies that $\left( \frac{1}{n} \right)^{(e)} = 0$ for all $1 \leq e \leq d$. Therefore, $L \geq d$. If $L = \infty$, then 
		\[ \fpt (f) = \frac{1}{a} + \frac{1}{b} + \frac{m - 2}{n}. \] Otherwise, 
		\[ \fpt (f) = \< \frac{1}{a} \> _{L} + \< \frac{1}{b} \>_{L} + (m - 2) \< \frac{1}{n} \> _{L} + \frac{1}{p^L}, \] which has denominator $p^L$ in lowest terms.
	\end{proof}
	
	Theorem \ref{slow} shows a linear asymptotic growth in the power of $p$ in the denominator of the $F$-pure thresholds. If we restrict to a subsequence, we can make the growth of the power of $p$ faster. 
	
	\begin{theorem}
		\label{fast}
		Let $s$ be a positive integer between 2 and $p - 1$ inclusive. For $n \ge 2$, let $f_n( \dott{x}{m}) = x_1^{a} + x_2^b + x_3^{p^n - s} + x_4^{p^n - s} + \dots + x_m ^{p^n - s}$, where $1 < a \leq b$ are integers and $(a, b) \neq (2, 2)$. Let $r = \left\lceil \log _{s} \left(  \left( 1 - \frac{1}{a} - \frac{1}{b} \right) (p - 1)  \right) \right\rceil$. If $p \equiv 1 \mod \lcm (a, b)$ and $s^r < p - 1$, then the power of $p$ in the denominator of $\fpt (f_n)$ is exactly $d \left( r + 1 \right) - 1$.
	\end{theorem}
	
	\begin{proof}
		Let $t = p^n - s$. We compute the nonterminating base $p$ expansion of $\frac{1}{t} = \frac{1}{p^n - s}$: $$\frac{1}{p^n - s} = \left(  \frac{1}{p^n} \right) \left( \frac{1}{1 - \frac{s}{p^n}} \right) = \sum_{q = 0}^{\infty} \frac{s^q}{p^{n (q + 1)}}.$$ Note that if $s^q \leq p - 1$, then $\left( \frac{1}{p^n - s} \right)^{(n (q + 1))} = s^q$ and $\left( \frac{1}{p^n - s} \right)^{(e)} = 0$ for $e < n (q + 1)$ and $d$ not dividing $e$. Now, 
		$$\left( \frac{1}{a} \right)^{(d(q + 1))} + \left( \frac{1}{b} \right)^{(n(q + 1))} + \left( \frac{1}{t} \right)^{(n(q + 1))} > p - 1 \iff q > \log _{s} \left( (p - 1)\left( 1 - \frac{1}{a} - \frac{1}{b} \right) \right).$$ 
		Therefore, we may choose $q = \left\lceil \log _{s} \left(  \left( 1 - \frac{1}{a} - \frac{1}{b} \right) (p - 1)  \right) \right\rceil$ (assuming $s^q < p - 1$), and so, $L = n \left( q + 1 \right) - 1$.
	\end{proof}
	
	\begin{remark}
		We consider the special case $m = 3$ and fix the family of diagonal hypersurfaces $f_n (x,y,z) = x^a + y^b + z^n$. Then Theorems \ref{stable} and either \ref{slow} or \ref{fast} take the following form: For any $p \equiv -1 \mod \lcm (a, b)$,
		then the denominator of $\fpt (f_n)$ is either $1$ or $p$, and $\fpt (f_n) = \fpt (f_p)$ for all $n \ge p$. For any $p \equiv 1 \mod \lcm (a, b)$ and $n = p^d - s$ for some integer $1 < s < p - 1$, then for $r = \left\lceil \log _{s} \left(  \left( 1 - \frac{1}{a} - \frac{1}{b} \right) (p - 1)  \right) \right\rceil$, if $s^r < p - 1$, then the power of $p$ in the denominator of $\fpt (f_n)$ is exactly $d \left( r + 1 \right) - 1$.
	\end{remark}
	
	\begin{cor}
		Given a prime $p > 5$ and positive integers $\ell, m$ with $m > 2$, there exists a family of diagonal hypersurfaces $\{ f_n \}_{n \in \N}$ in $m$ variables such that the power of $p$ in the denominator of $\fpt (f_n)$ is less than or equal to $\ell$ for all $n \in \N$, and $\fpt (f_n) = \fpt (f_{p^{\ell}})$ for all $n \ge p^{\ell}$. There also exists a different family of diagonal hypersurfaces $\{ g_n \}_{n \in \N}$ such that the power of $p$ in the denominator of $\fpt (g_n)$ diverges to $\infty$ as $n \to \infty$.
	\end{cor}
	
	\begin{proof}
		The case where $\ell = 1$ is done by Theorems \ref{stable} and either \ref{slow} or \ref{fast}. For the case $\ell > 1$, if the family $f_n (\dott{x}{m}) = x_1^{a_1} + x_2^{a_2} + \dots + x_{m - 1}^{a_{m - 1}} + x_m^{n}$ satisfies Theorem \ref{stable}, then consider the family $g_n (\dott{x}{m}) = x_1^{p^{\ell} a_1} + x_2 ^{p^{\ell} a_2} + \dots+ x_{m - 1}^{p^{\ell} a_{m - 1}} + x_m ^{p^{\ell}n}$. Since the nonterminating base $p$ expansion of $\frac{1}{p^{\ell}a}$ for any positive integer $a$ satisfies $\left( \frac{1}{a} \right)^{(e)} = \left( \frac{1}{p^{\ell} a} \right)^{(e + \ell)}$ for all positive integers $e$, we may then proceed similarly to the proof of Theorem \ref{stable}.
	\end{proof}
	
	We end with some computational remarks. All computations are made using the \textit{FrobeniusThresholds} package of Macaulay2.
	
	\begin{remark}
		It is possible for the power of $p$ in the denominator of $\fpt (f_n)$ for some $n$ between $p^d$ and $p^{d + 1}$ to be higher than the prescribed power in Theorem \ref{fast}. For example, let $a = 2, b = 3, p = 97, d = 2, n = 9216$. Then the power of $p$ in the denominator of $\fpt (x^2 + y^3 + z^{9216})$ is $16$, but the power of $p$ in the denominator of $\fpt (x^2 + y^3 + z^{97^2 - s})$ for $2 \le s \le 96$ is at most $11$.
	\end{remark}
	
	\begin{remark}
		Theorem \ref{stable} may inspire the conjecture that the denominator of $\fpt (x_1^{a_1} + x_2^{a_2} + \dots + x_{m - 1}^{a_{m - 1}} + x_m^{n})$ for $p \equiv -c \mod \text{lcm } (a, b)$ and $c \leq \frac{p - 1}{2}$ is $p^e$ for some positive integer $e \leq c$.  Unfortunately, this is false. For example, let $a = 3, b = 5, n = 699, p = 71$: note that $p \equiv -4 \mod 15$, but $\fpt (x^a + y^b + z^n) = \frac{964836811}{71^5}$.
	\end{remark}
	
	\begin{remark}
		The growth of the denominator in Theorems \ref{slow} and \ref{fast} may not hold if $p \equiv c \mod \lcm (a, b)$ and $c \le \frac{ p - 1}{2}$. For example, let $a = 3, b = 5, n = 500, p = 17$: note that $p \equiv 2 \mod 15$, and the nonterminating base $p$ expansions of $\frac{1}{3}$ and $\frac{1}{5}$ are $\frac{1}{3} = . 5 : 11 : \dots$ and $\frac{1}{5} = .3:6: \dots,$ giving $\fpt (x^3 + y^5 + z^{500}) = \left< \frac{1}{3} \right>_{1} + \left< \frac{1}{5} \right>_{1} + \left< \frac{1}{500} \right>_{1} + \frac{1}{17} = \frac{9}{17}.$
	\end{remark}
	
	\begin{remark}
		(Necessity of $s^r < p - 1$) For $a = 3, b = 4, n = 165, p = 13$: By Macaulay2's computations, $\fpt (x^a + y^b + z^n) = \frac{480786743}{13^8}$. Note that $n = p^2 - 4$, and the nonterminating base $p$ expansions of $\frac{1}{3}$ and $\frac{1}{4}$ are $\frac{1}{3} = .\overline{4}$ and $\frac{1}{4} = .\overline{3}$. Now, for $r = 1$, $4 + 3 + s^r = 4 + 3 + 4 = 11 < p$. However, for $r = 2$, $s^r = 16 > p - 1$, so the ``predicted" power would be $5$, which is wrong.
	\end{remark}

\end{document}